\documentclass[10pt, letterpaper]{IEEEtran}  
\IEEEoverridecommandlockouts
\usepackage{mathtools} 
\usepackage{hyperref,graphicx} 
\usepackage{amsmath} 
\usepackage{amssymb}  
\usepackage{cite}
\usepackage{multirow}
\usepackage{color}
\usepackage{algorithm, algorithmicx, algpseudocode}

\newcommand{\B}{{\mathbb B}}
\newcommand{\D}{{\mathbb D}}
\newcommand{\A}{{\mathbb A}}
\newcommand{\R}{{\mathbb R}}

\newcommand{\N}{{\mathbb N}}
\newcommand{\Rnn}{{\mathbb R}_{\ge 0}}

\newcommand{\C}{{\mathbb C}}

\newcommand{\cA}{{\mathcal A}}

\newcommand{\cB}{{\mathcal B}}

\newcommand{\cD}{{\mathcal D}}

\newcommand{\cK}{{\mathcal K}}

\newcommand{\cM}{{\mathcal M}}

\newcommand{\cO}{{\mathcal O}}
\newcommand{\cP}{{\mathcal P}}

\newcommand{\diag}{{\mathrm{diag}}}

\newcommand{\cl}{\mathrm{cl}}

\newcommand{\inter}{\mathrm{int}}

\def\QED{\mbox{\rule[0pt]{1.3ex}{1.3ex}}} 

\newenvironment{proof}{{\quad \it Proof:\,}}{\hfill \QED \par}

\newenvironment{proof-of}[1]{{\it Proof of #1:\,}}{\hfill\QED\par}
\newtheorem{thm}{Theorem}

\newtheorem{defn}{Definition}

\newtheorem{prop}{Proposition}

\title{Properties of Isostables and Basins of Attraction of Monotone Systems}
\author{Aivar Sootla and Alexandre Mauroy
\thanks{Aivar Sootla is with Montefiore Institute, University of Li\`{e}ge, B-4000, Belgium {\tt asootla@ulg.ac.be}. 
Alexandre Mauroy is with the Luxembourg Centre for Systems Biomedicine, University of Luxembourg, L-4367 Belvaux, Luxembourg {\tt {alexandre.mauroy@uni.lu}}}
\thanks{A. Sootla holds an F.R.S--FNRS fellowship. Most of this work was performed when A. Mauroy was with the University of Li\`{e}ge and held a return grant from the Belgian Science Policy (BELSPO).}}
\begin{document}
\maketitle
\thispagestyle{empty}
\pagestyle{empty}
\begin{abstract}
In this paper, we investigate geometric properties of monotone systems by studying their isostables and basins of attraction. Isostables are boundaries of specific forward-invariant sets defined by the so-called Koopman operator, which provides a linear infinite-dimensional description of a nonlinear system. First, we study the spectral properties of the Koopman operator and the associated semigroup in the context of monotone systems. Our results generalize the celebrated Perron-Frobenius theorem to the nonlinear case and allow us to derive geometric properties of isostables and basins of attraction. Additionally, we show that under certain conditions we can characterize the bounds on the basins of attraction under parametric uncertainty in the vector field. We discuss computational approaches to estimate isostables and basins of attraction and illustrate the results on two and four state monotone systems.
\end{abstract}
\section{Introduction}
In many applications, such as economics~\cite{leontief1986input} or biology~\cite{hirsch2005monotone}, linear dynamical systems have states, which take only nonnegative values. For example, protein concentrations in biology are always nonnegative. These systems are called \emph{positive} and have received a considerable attention in the context of systems theory~\cite{coxson1987positive},~\cite{PosSysBook}, model reduction~\cite{Sootla2012positive, grussler2012symmetry}, distributed control~\cite{rantzer2015ejc},~\cite{tanaka2011bounded}, etc. Analysis of such systems is facilitated by employing the celebrated Perron-Frobenius theorem (cf.~\cite{berman1994nonnegative}), which establishes strong spectral properties of the drift matrix in a linear positive system. 

In the nonlinear setting positive systems have a couple of generalizations, with the most known being \emph{cooperative monotone systems} (cf.~\cite{hirsch2005monotone}). More specifically positive systems generalize to cooperative with respect to the positive orthant systems. Similarly to the linear case, these nonlinear systems generate trajectories (or flows) which for every time $t$ are increasing functions in every argument with respect to the initial state. With a slight abuse of notation, we will refer to cooperative monotone systems simply as \emph{monotone}. In~\cite{hirsch2005monotone}, it was briefly mentioned that the flow of a monotone system can be seen as a positive operator. Hence the authors argued that an operator extension of the Perron-Frobenius theorem, which is called the Krein-Rutman theorem~\cite{krein1950positive}, can be applied. However, the investigation into spectral properties of these operators lacked due to absence of a well-developed theory of spectral elements of such operators. This gap was filled by the development of the so-called Koopman operator and the associated semigroup (cf.~\cite{mezic2005, budivsic2012applied, mezic2013analysis}). 

The Koopman semigroup is a linear infinite-dimensional representation of a nonlinear dynamical system. In many practical situations, the spectral elements of this semigroup such as eigenfunctions (which can be seen as infinite dimensional eigenvectors) and eigenvalues can be computed~\cite{mauroy2013isostables,mauroy2014global}. As in the linear case, the dominant eigenfunctions (i.e., eigenfunctions corresponding to the eigenvalue with the largest real part) offer an insight into the dynamics of the system. In~\cite{mauroy2013isostables}, it is shown that the level sets of the dominant eigenfunction (called \emph{isostables}) contain the initial states of trajectories that converge synchronously toward the equilibrium.  The level set at infinity is the boundary of the basin of attraction of the equilibrium. To summarize, the dominant eigenfunctions and eigenvalues contain the information about the asymptotic behavior of the system.

In this paper, we first study the properties of the Koopman semigroup associated with a monotone system. We show that for stable systems monotone with respect to the nonnegative orthant the dominant eigenfunction is an increasing function in every argument for all states in the basin of attraction. Hence we offer yet another version of the Perron-Frobenius theorem for nonlinear systems (see~\cite{lemmens2012nonlinear} for other results on the subject). We use our version of the Perron-Frobenius theorem to study geometric concepts of isostables and basins of attraction for monotone systems. Using this statement we provide a straightforward proof of a known result stating that a basin attraction of a monotone system is a union of the so-called \emph{order-intervals} (cf.~\cite{smith2001stable, jiang2004saddle,ruffer2010computing}). We refine this result by showing that the sublevel sets of the dominant eigenfunctions are unions of order-intervals, as well.

We proceed our theoretical development by considering the basins of attraction of bistable monotone systems under parametric uncertainty in the vector field.  We show that if the system is monotone with respect to parameter variations, then it is possible to estimate inner and outer bounds on the basin of attraction of the uncertain system. In this instance, instead of the spectral theory we use monotone control systems theory. Moreover, it appears that the behavior of the eigenfunctions and isostables under parametric uncertainty is more complicated in comparison with the behavior of basins of attraction.

The rest of the paper is organized as follows. In Section~\ref{s:prel}, we cover the main properties of the Koopman operator and monotone systems. In Subsection~\ref{s:spectral} we obtain spectral properties of monotone systems, while deriving some properties of the isostables in Subsection~\ref{ss:iso}. We investigate the behavior of basins of attraction of monotone systems under parameter variations in Subsection~\ref{s:basin}. We also discuss different algorithms to compute isostables and the boundary of the basins of attraction in Section~\ref{s:alg}. We conclude the paper with numerical examples in Section~\ref{s:examples}. 

\section{Preliminaries\label{s:prel}}
Throughout the paper we consider parameter-dependent systems in the following form
\begin{equation}
\label{sys:f}
\dot x = f(x,p),\quad x(0) = x_0,
\end{equation} 
where $f: \cD\times \cP\rightarrow \R^n$, where $\cD\subset\R^n$ and $\cP\subset\R^m$ for some integers $n$ and $m$. We define the flow map $\phi_f: \R \times \cD \times \cP\rightarrow \R^n$, where $\phi_f(t, x_0, p)$ is a solution to the system~\eqref{sys:f} with an initial condition $x_0$ and a parameter $p$. We assume that $f(x,p)$ is continuous in $(x,p)$ on $\cD\times \cP$ and Lipschitz in $x$ on a every compact subset of $\cD$ for every fixed $p$. When it is clear from the context, we will drop the parameter-dependence.

\subsection{Monotonicity}
We will study the properties of the system \eqref{sys:f} with respect to a partial order induced by cones. A set $\cK$ is called \emph{a positive cone} if it is closed under addition and multiplication by a nonnegative scalar and if we have $-x\not\in \cK$ for any $x\in \cK$. A relation $\sim$ is called a {\it partial order} if it is reflexive ($x\sim x$), transitive ($x\sim y$, $y\sim z$ implies $x\sim z$), and antisymmetric ($x\sim y$, $y\sim x$ implies $x = y$). We define a partial order $\succeq_\cK$ through a cone $\cK\subset\R^n$ as follows: $x\succeq_\cK y$ if and only if $ x - y \in \cK$. We will also write $x\succ_\cK y$ if $x\succeq_\cK y$ and $x\ne y$, and $x\gg_\cK y$ if $x- y \in \inter(\cK)$. We say that the order is standard if $\cK = \Rnn^n$, which, with a slight abuse of notation, we denote as $\succeq$ without a subscript. We call a set $\{ z | x\preceq_\cK z \preceq_\cK y \}$ \emph{an order-interval} induced by a cone $\cK$ and denote it as $[x,~y]_\cK$. A set $\cA$ is called \emph{p-convex} if for every $x$, $y$ in $\cA$ such that $x\succeq_\cK y$, and every $\lambda\in(0,1)$ we have that $\lambda x+ (1-\lambda) y\in \cA$. 
Systems whose flows preserve a partial order relation $\succeq_\cK$ are called \emph{monotone systems}. 

\begin{defn}\label{def:mon}
The system is \emph{monotone} with respect to the cones ${\cK_x}$, ${\cK_p}$ if $\phi_f(t, x, p)\preceq_{\cK_x} \phi_f(t, y, q)$ for all $t\ge 0$, and for all $x\preceq_{\cK_x} y$ and $p \preceq_{\cK_p} q$. The system is \emph{strongly monotone} with respect to the cones ${\cK_x}$, ${\cK_p}$ if it is monotone and $\phi_f(t, x, p) \ll_{\cK_x} \phi_f(t, y, q)$ holds for all $t>0$ provided $x\preceq_{\cK_x} y$, $p \preceq_{\cK_p} q$, and either $x\prec_{\cK_x} y$ or $p \prec_{\cK_p} q$ holds.
\end{defn}

A certificate for monotonicity with respect to an orthant is called \emph{Kamke-M\"uller} conditions~\cite{angeli2003monotone}, where some generalizations of this result may also be found.
\begin{prop}[\cite{angeli2003monotone}]\label{prop:kamke}
Consider the system~\eqref{sys:f}, where $f$ is differentiable in $x$ and $p$ and let the sets $\cD$, $\cP$ be p-convex. Then the system~\eqref{sys:f} is monotone on $\cD\times \cP$ with respect to the standard partial orders if and only if 
\begin{gather*}
\frac{\partial f_i}{\partial x_j}\ge 0,\quad\forall~i\ne j,\quad(x,p)\in\cl(\cD)\times\cP\\
\frac{\partial f_i}{\partial p_j}\ge 0,\quad\forall~i, j,\quad(x,p)\in\cD\times\cP.
\end{gather*}
\end{prop}
\subsection{Koopman Operator} 
Spectral properties of nonlinear dynamical systems can be described through an operator-theoretic framework that relies on the so-called Koopman operator. The Koopman operator associated with $\dot x = f(x)$ is an operator acting on the functions $g:\R^n\rightarrow \C$ (also called observables). The Koopman operator generates the semigroup
\begin{gather}
U^t g(x) = g \circ \phi_f(t,x),
\end{gather}
where $\circ$ is the composition of functions and $\phi_f(t,x)$ is a solution to the considered system. The Koopman semigroup is linear~\cite{mezic2005}, hence it is natural to study its spectral properties. In particular, the eigenfunctions $s_j(x)$ of the Koopman semigroup are defined as the functions $\R^n\rightarrow \C$ satisfying
\begin{gather}
\label{eq:property_eigenf}
U^t s_j(x) = s_j(\phi_f(t, x)) = s_j(x) \, e^{\lambda_j t},
\end{gather}
and $\lambda_j\in \mathbb{C}$ is the associated eigenvalue. We can also obtain a very useful expression
\begin{gather}
f(x)^T \nabla s_j(x) = \lambda_j s_j(x). \label{eq:s-one}
\end{gather}
If $f$ is analytic, and if the eigenvalues $\lambda_j$ of the Jacobian of the vector field at the equilibrium $x^\ast$ are distinct and $\Re(\lambda_j)<0$ for all $j$, then the flow of the system can be expressed (at least locally) through the following expansion:
\begin{align}
\label{eq:expansion}&\phi_f(t, x) = x^\ast + \sum\limits_{j=1}^{n} s_j(x) v_j e^{\lambda_j t} + \\
&\sum\limits_{\begin{smallmatrix}
k_1,\dots,k_n\in\N_0\\
k_1+\dots+k_n>1 
\end{smallmatrix}} v_{k_1,\dots,k_n} \, s_1^{k_1}(x) \cdots s_n^{k_n}(x) e^{(k_1\lambda_1 +\dots k_n\lambda_l) t}, \nonumber
\end{align}
where $\N_0$ is the space of nonnegative integers, $v_j$ are the right eigenvectors corresponding to $\lambda_j$, the vectors $v_{k_1,\dots,k_n} \,$ are the so-called Koopman modes (cf.~\cite{mezic2005, mauroy2014global} for more details). In the case of a linear system $\dot x = A x$ with matrix $A$ having the left eigenvectors $w_i$, the eigenfunctions $s_i(x)$ are equal to $w_i^T x$ and the expansion~\eqref{eq:expansion} has only the finite sum. A similar (but lengthy) expansion can be obtained even if the eigenvalues $\lambda_j$ are not distinct and have linearly dependent eigenvectors (cf.~\cite{mezic2015applications}).

Let $\lambda_j$ be such that $0>\Re(\lambda_1)\ge \Re(\lambda_j)$, $j\neq 1$, then the eigenfunction $s_1$, which we call \emph{dominant}, can be computed using the so-called Laplace average (cf.~\cite{mauroy2013isostables})
\begin{gather}\label{laplace-average}
g_{\lambda_1}^\ast(x) = \lim\limits_{t\rightarrow \infty}\frac{1}{t}\int\limits_0^t (g\circ \phi_f(s, x)) e^{-\lambda_1 s} d s.
\end{gather}
For all $g$ that satisfy $g(x^\ast)=0$ and $v_1^T \nabla g(x^\ast) \neq 0$, the Laplace average $g_{\lambda_1}^\ast$ is equal to $s_1(x)$ up to a multiplication with a scalar. Therefore, without loss of generality, we will only consider so-called increasing functions as observables. A function $g(x):\R^n \rightarrow \R$ is called \emph{increasing} with respect to $\cK$ if for all $x$, $z$ such that $x\preceq_\cK z$ we have $g(x) \le g(z)$. The eigenfunctions $s_j(x)$ with $j\ge 2$ (non-dominant eigenfunctions) are generally harder to compute and are not considered in the present study. 

The eigenfunction $s_1(x)$ captures the dominant (i.e. asymptotic) behavior of the system. Moreover for $\lambda_1 \in \R$, it follows from~\eqref{eq:expansion} that the trajectories starting from the boundary $\partial \cB_\alpha$ of the set $\cB_\alpha =  \{x | |s_1(x)| \le \alpha \}$ share the same asymptotic evolution
\begin{equation*}
\phi_f(t,x) \rightarrow x^\ast + v_1 \, \alpha  e^{\lambda_1 t}\,, \quad t\rightarrow \infty\,.
\end{equation*}

Naturally, these sets are important for understanding the dynamics of the system. The sets $\partial\cB_\alpha=\{x| |s_1(x)|=\alpha\}$ are called \emph{isostables}, and contain the initial conditions of trajectories that converge synchronously toward the equilibrium. In the neighborhood of $x^\ast$, the isostables are parallel hyperplanes (if $\lambda \in \R$).  We will also employ the notations 
\begin{gather*}
\partial_+ \cB_\alpha=\left\{x\Bigl| s_1(x) = \alpha\right\},\quad
\partial_- \cB_\alpha=\left\{x\Bigl| s_1(x) = -\alpha\right\}
\end{gather*} 
for real and nonnegative $\alpha$. A more rigorous definition of isostables and more information can be found in~\cite{mauroy2013isostables}.

\section{Geometric Properties of Monotone Systems \label{s:geo-basins}}
\subsection{Spectral Properties of Monotone Systems \label{s:spectral}}
We start by establishing the properties of the dominant eigenfunctions and eigenvalues of monotone systems.
\begin{prop} \label{prop:mon-eig-fun}
Assume that the system $\dot x = f(x)$ admits a stable equilibrium $x^\ast$ with a basin of attraction $\cB$. Let $\lambda_i$ be the eigenvalues of the Jacobian at $x^\ast$ such that $0>\Re(\lambda_i) \ge \Re(\lambda_j)$ for all $i \le j$, and let $v_j$ be  corresponding right eigenvectors. Let $s_1(\cdot)\in C^1(\cB)$ be the dominant eigenfunction with $v_1^T \nabla s_1(x^\ast) = 1$. Then:

(i) if the system is monotone with respect to $\cK$ on $\mathrm{int}(\cB)$, then $\lambda_1$ is real and negative, $s_1(x) \ge s_1(y)$ for all $x$, $y$ satisfying $x\succeq_\cK y$, and $v_1\succ_\cK 0 $. Additionally, we have that $s_1(x) > s_1(y)$ for all $x \gg_\cK y$.

(ii) if the system is strongly monotone with respect to $\cK$ on $\mathrm{int}(\cB)$, then  $\lambda_1$ is simple, real and negative, $s_1(x)> s_1(y)$ for all $x$, $y$ satisfying $x\succ_\cK y$, and $v_1\gg_\cK 0$.
\end{prop}

Both conditions (i) and (ii) are only necessary and not sufficient, which is consistent with the linear case. The proof of this result is similar to the proof of the main result in~\cite{sootla2015koopman} concerning spectral properties of the so-called eventually monotone systems. If the vector field is analytic, then the second condition is necessary and sufficient for so-called strong eventual monotonicity. 

We can view this result through the prism of positive operator theory. An operator $A: V\rightarrow V$ acting on a normed space $V$ is called \emph{positive} (or invariant) with respect to a cone $\cK\subseteq V$, if $A \cK \subseteq \cK$. It can be shown that the Koopman semigroup $U^t$ associated with a monotone dynamical system is a positive operator with respect to the cone of increasing functions at every time $t$. Hence the following result may be seen as a nonlinear version of the Perron-Frobenius theorem or a version of the Krein-Rutman theorem for the Koopman operator. We note that a semigroup associated with a nonlinear system was studied from the positive operator viewpoint in the context of diffusion processes in~\cite{herbst1991diffusion}. 

There were other successful attempts to extend the Perron-Frobenius theorem to nonlinear systems (cf.~\cite{gaubert2004perron, lemmens2012nonlinear, aeyels2002extension}), which however are based on applications of Perron-Frobenius arguments to nonlinear maps, while our approach is rather based on their infinite-dimensional extension. Moreover, our main contribution is that our version uses eigenfunctions, which give additional insight into the dynamics of the system and can be estimated using data-based, algebraic or simulation methods (cf.~\cite{mauroy2014global}). Nevertheless, there are similarities, and an investigation into a connection between these results is one of the future work directions.

\subsection{Isostables of Monotone Systems}
\label{ss:iso}
First, we recall a known result on the geometry of basins of attraction $\cB$ and their boundary $\partial \cB$. This result is briefly mentioned in~\cite{hirsch2005monotone,jiang2004saddle, smith2001stable} in the case of bistable monotone systems. The following version requires monotonicity on the basin of attraction only (for a complete proof see e.g.~\cite{sootla2015pulsesaut}).

\begin{prop} \label{prop:boundary}
Let the system $\dot x = f(x)$ have an asymptotically stable equilibrium $x^\ast$ with a domain of attraction $\cB$ and be monotone on $\cB$ with respect to $\cK$. Then for any $x$, $y\in\cB$ the order-interval $[x,~y]_\cK$ is a subset of $\cB$. 
\end{prop}

We proceed by showing similar properties of the sets $\cB_\alpha = \{x | |s_1(x)| \le \alpha \}$ and isostables $\partial \cB_\alpha = \{x | |s_1(x)| = \alpha\}$. 

\begin{prop} \label{prop:level-set-proper}
Let the system $\dot x = f(x)$ have a stable hyperbolic equilibrium $x^\ast$ with a domain of attraction $\cB$ and be monotone with respect to $\cK$ on $\cB$ with $s_1(\cdot)\in C^1(\cB)$. Then the following statements hold:

(i) for any $x$, $y\in\cB_\alpha\bigcap \cB$, the order-interval $[x,~y]_\cK$ is a subset of $\cB_\alpha\bigcap \cB$,

(ii) the manifolds $\partial_+ \cB_\alpha$, $\partial_- \cB_\alpha$ do not contain $x$, $y$ such that $x\gg_\cK y$; 

(iii) if the system is strongly monotone then the manifold $\partial \cB_\alpha$ does not contain $x$, $y$ such that $x\succ_\cK y$.
\end{prop}

The proof of this result is identical to the case of (strongly) eventually monotone systems~\cite{sootla2015koopman}. It is important to note that the boundary $\partial\cB$ can contain two points $x$, $y$ such that $x\gg_\cK y$, but in this case these points belong to different manifolds $\partial_- \cB_\alpha$, $\partial_+ \cB_\alpha$. The point (ii) can also be shown for $\alpha = \infty$, hence the boundaries $\partial_-\cB_\infty$, $\partial_+\cB_\infty$ do not contain $x$, $y$ such that $x\gg_\cK y$ as well. The point (iii) does not generalize in a straightforward manner to the case $\alpha=\infty$. This is however shown in~\cite{smith2001stable,jiang2004saddle} under additional assumptions.

\subsection{Basins of Attraction of Monotone Systems Subject to Parameter Variations \label{s:basin}}
First we need to discuss the following assumptions.
\begin{enumerate}
\item[A1.] The system $\dot x = f(x,p)$ is a bistable system on $\cD$ with two stable equilibria $x^\ast(p)$, $x^{\bullet}(p)$ for all $p$. Assume also that the system is monotone in $x$ and $p$.
\item[A2.] Parameters $p$ take values from a compact set $\cP$ with vectors $p_{\rm max}$, $p_{\rm min}\in \cP$ such that 
\[
p_{\rm min} \preceq p \preceq p_{\rm max} \quad\forall p\in\cP.
\]
\item[A3.] The following holds for all $p \in \cP$: 
\begin{gather*} 
x^\ast(p) \in \bigcap\limits_{q\in\cP} \cB(x^{\ast}(q)),  \quad x^{\bullet}(p) \in \bigcap\limits_{q\in\cP} \cB(x^{\bullet}(q))
\end{gather*}
\item[A4.]  $x^{\bullet}(p_{\rm min}) \gg x^{\ast}(p_{\rm max})$.
\end{enumerate}

Assumption~{A1} defines a monotone bistable system, which is also monotone with respect to parameter variations. In order to simplify the presentation we consider only the standard orders for monotonicity in the initial conditions and parameters (which can be verified by means of Proposition~\ref{prop:kamke}), but our results hold for other orderings as well. In order to obtain any meaningful bounds on the basins of attraction, we need to assume that the set $\cP$ is compact and ordered, which is done in Assumption~{A2}. We also assume that parameter variations are small enough so that no bifurcation occurs under these parameter variations. Moreover, we assume that the parameter variations do not affect the locations of the equilibria in a significant way, which is done in Assumption~{A3}. This assumption and Assumption~{A4} are technical and perhaps can be avoided\footnote{For example, we can modify Assumption~A3 to require the sets $\bigcap\limits_{q\in\cP} \cB(x^{\ast}(q))$, $\bigcap\limits_{q\in\cP} \cB(x^{\bullet}(q))$ to be non-empty, however, this modified assumption is harder to check}, but they are straightforward to satisfy and will simplify the presentation of results. The assumptions on the location of stable equilibria and the basins of attraction are schematically depicted in Figure~\ref{fig:ra-ss}.  Assumptions~{A1, A2, and  A4} are straightforward to check, while Assumption~{A3} can only be verified after basins of attraction for $p_{\rm min}$ and $p_{\rm max}$ are computed. Using these assumptions we can derive the following result. 

\begin{figure}[t]\centering
  \includegraphics[width = 0.65\columnwidth]{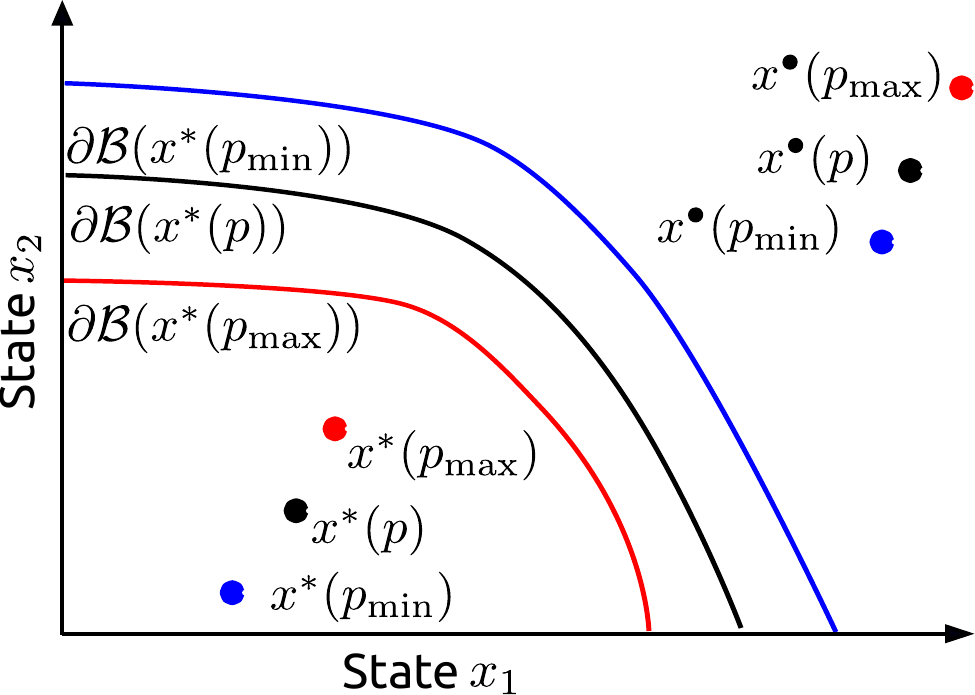}
\caption{A schematic depiction of Assumptions~{A3} and {A4}.} \label{fig:ra-ss}
\end{figure}
\begin{thm} \label{thm:ba}
Consider the system $\dot x = f(x,p)$ under the assumptions {A1 -- A4}. Then $\cB(x^\ast(p_1)) \subseteq \cB(x^\ast(p_2))$ and $\cB(x^{\bullet}(p_1)) \supseteq \cB(x^{\bullet}(p_2))$ for all $p_1 \succeq p_2$. In particular, we have 
\begin{gather}
\begin{gathered}
\cB(x^{\ast}(p_{\rm min})) \subseteq \cB(x^{\ast}(p)) \subseteq \cB(x^{\ast}(p_{\rm max}))~\forall p\in \cP \\
\cB(x^{\bullet}(p_{\rm min})) \supseteq \cB(x^{\bullet}(p)) \supseteq \cB(x^{\bullet}(p_{\rm max}))~\forall p\in \cP
\end{gathered}\label{thm:main-res}
\end{gather}
\end{thm}
\begin{proof}
i) We need to show that for all $p_1$, $p_2$ we have that $x^{\bullet}(p_1) \gg x^{\ast}(p_2)$. But first we will verify that $x^{\ast}(p_1) \succeq x^{\ast}(p_2)$ for all $p_1 \succeq p_2$. Due to monotonicity for all $p_1\succeq p_2$ we have that 
\[
\phi(t, x^\ast(p_1), p_1) \succeq \phi(t, x^\ast(p_1), p_2),
\]
where $\phi(t, x^\ast(p_1), p_1)$ is equal to $x^\ast(p_1)$ for all $t$, while  $\phi(t, x^\ast(p_1), p_2)$ converges to $x^\ast(p_2)$ since $x^\ast(p_1) \in \cB(x^{\ast}(p_2))$ due to Assumption~{A3}. Hence, $x^\ast(p_1) \succeq x^\ast(p_2)$ for all $p_1$, $p_2\in\cP$ such that $p_1 \succeq p_2$. Similarly we can show that $x^{\bullet}(p_1) \succeq x^{\bullet}(p_2)$ for all $p_1 \succeq p_2$. Together with Assumption~{A4}, this implies that $x^{\bullet}(p_1) \gg x^{\ast}(p_2)$ for all $p_1$, $p_2\in\cP$.

ii) Let $y\in\cB(x^\ast(p_1))$, due to monotonicity we have that 
\[
\phi(t, y, p_1) \succeq \phi(t, y, p_2),
\]
for $p_1 \succeq p_2$, where $\phi(t, y, p_1)$ converges to $x^\ast(p_1)$ with $t\rightarrow \infty$. Now, by point i), we conclude that $\phi(t, y, p_2)$ cannot converge to $x^{\bullet}(p_2)$, hence it converges to $x^\ast(p_2)$. This implies that $\cB(x^\ast(p_1)) \subseteq \cB(x^\ast(p_2))$  for all $p_1 \succeq p_2$.

iii) The inclusion $\cB(x^{\bullet}(p_1)) \supseteq \cB(x^{\bullet}(p_2))$ follows from the fact that $\cl(\cB(x^{\bullet}(p))) =\cl(\cD \backslash \cB(x^\ast(p)))$ for all $p\in\cP$. While~\eqref{thm:main-res} is now straightforward to show.
\end{proof}

Theorem~\ref{thm:ba} implies that we can predict a direction of change in the basins of attraction subject to parameter variations if we check a couple of simple conditions. In many applications (e.g. synthetic biology), so-called toggle switches are used, which are bistable systems. The larger the domain of attraction, the more robust is the toggle switch with respect to the corresponding equilibrium. Therefore this result can be valuable for design purposes in applications, where we need to ensure operational robustness of a stable equilibrium in a toggle switch subject to intrinsic and/or exogenous noise.

\section{Tools for Computing Isostables and Basins of Attraction \label{s:alg}}
In control theory, a go-to approach to compute forward-invariant sets (not only basins of attraction) of dynamical systems is sum-of-squares (SOS) programming. This approach can be applied to systems with polynomial vector fields and also can provide estimates on the basins of attraction for such systems. There exists a number different SOS-based methods for estimation of basins of attraction, see e.g.~\cite{henrion2014convex} and~\cite{valmorbida2014region} and the references within.

Another option is to compute the function $s_1(\cdot)$, which provides the isostables and the basin of attraction. In the case of a polynomial vector field, we can formulate the computation of $s_1(x)$ as an infinite dimensional \emph{linear algebraic} problem using~\eqref{eq:s-one}. Hence, we can provide an approximation of $s_1(x)$  using linear algebra by parametrizing $s_1(x)$ with a finite number of basis functions (cf.~\cite{mauroy2014global}). On another hand, we can estimate $s_1(x)$ \emph{directly from data} using dynamic mode decomposition methods (cf.~\cite{Schmid2010,Tu2014}). These two options provide extremely cheap estimates of $s_1$. In fact, the algebraic methods (as demonstrated in~\cite{mauroy2014global}) also provide good estimates on basins of attraction. However, we cannot typically compute estimates with an excellent approximation quality, which comes as trade-off for fast computations.

If the quality of approximation is very important we can simulate a number of trajectories with initial points on a mesh grid of the state-space and compute $s_1(\cdot)$ in these points using Laplace averages~\eqref{laplace-average}. After that different interpolation or machine learning methods can be applied to estimate the dominant eigenfunction.

We stress that computing the eigenfunction using Laplace averages is not necessarily more computationally difficult than computing the basin of attraction using sum-of-squares. Moreover, the Laplace average method can be applied to systems with non-polynomial vector fields and is well-suited to high dimensions. This is due to the fact that the simulation time of a nonlinear system can significantly be reduced by model reduction (time-scale separation) methods. In contrast to SOS-based methods, however, the Laplace average method does not generally provide a certificate, that is we cannot guarantee that the estimated set is an inner or an outer approximation of a basin of attraction. 
 
In this paper, we use Laplace averages in our computational algorithm. We use an adaptive grid and thus lower computational complexity by exploiting the facts that the eigenfunction is increasing and only one level set is needed. Finally, monotonicity allows us to guarantee that we compute an inner and an outer approximation of the basin, hence providing a certificate. 

\subsection{Algorithm for Computation of a Level Set of an Increasing Function}
First we derive an algorithm for computation of the level sets of an arbitrary increasing with respect to $\cK$ function $y = g(z)$. Our algorithm can be extended to an arbitrary positive cone $\cK$ in $\R^n$, however, in order to simplify the presentation we focus on the ordering induced by a positive orthant. We compute the level sets within a hypercube $\B = \{z | b^0 \preceq z \preceq b^1 \}$. In particular, we compute a set of points $\cM^{\rm min}$, which lies in $\{z |g(z) < \alpha\}$, and a set of points  $\cM^{\rm max}$, which does not lie in $\{z |g(z) < \alpha\}$. The set $\cM^{\rm min}$ (respectively, $\cM^{\rm max}$) will not contain points $x$, $y$ such that $x\gg y$. Since $g(z)$ is an increasing function this will allow to build a piecewise constant inner approximation (respectively, an outer approximation) of the sublevel set. In order to do so we will use the following oracle
\begin{gather}
\cO(z) = \begin{cases}
0  & \textrm{ if }g(z) < \alpha, \\
1  & \textrm{ otherwise}.
\end{cases} 
\end{gather}

We will generate the points on the face of $\B$ with $z_n =b^0_n$ randomly, while choosing the point $z_n$ greedily. In order to do so efficiently, we need to modify a hypercube $\B$ in such a way that the curve $y = g(z)$ does not intersect the faces $z_n = b^0_n$ and $z_n = b^1_n$. Therefore, first we need to increase $b^1_n$ sufficiently such that all vertices of the face with $z_n = b^1_n$ return the value of the oracle $\cO$ equal to one. At the same time, we need to lower the face with $z_n = b^0_n$ significantly so that all the vertices of this face return $\cO$ equal to zero. Due the constraints on the function we can have $b^0 =0$ and $z \succeq 0$, which is typically the case in the state-space of biological applications. In this case, we need to move the points $b^0_1$, $\dots$ $b^0_{n-1}$ and shrink $\B$ in order to achieve the same goal. We can do so, for example, by computing the level set $g(z) = \alpha$ with $z_n = b^0_n$ and a pick a point with the smallest norm. 

Our algorithm exploits the monotonicity property of $g$ in the following manner. If a sample $z^j$ is such that $\cO(z^j) = 0$, then for all $w \preceq z^j$ we have $\cO(w)  = 0$. Similarly, if a sample $z^j$ is such that $\cO(z^j) = 1$, then for all $w \succeq z^j$ we have $\cO(w)  = 1$. Therefore, we need to keep track of the largest (in the order) samples $z^j$ with $\cO(z^j) =0$, the set of which we denote $\cM^{\rm min}$, and the smallest (in the order) samples $z^j$ with $\cO(z^j) = 1$, the set of which we denote $\cM^{\rm max}$. Hence, in order to approximate the function $g(\cdot)$ at every step we generate new samples $z^j$ such that $\cM^{\rm min} \preceq z^j \preceq \cM^{\rm max}$. After generating the samples we compute the oracle values $\cO(z^j)$ and add new samples to the set $\cM^{\rm min}$ if $\cO(z^j)=0$, or to the set $\cM^{\rm max}$ if $\cO(z^j) = 1$. After new samples have been added there might exist $w$, $z$ in $\cM^{\rm min}$ (respectively, $\cM^{\max}$) such that $w\preceq z$ (respectively, $w \succeq z$). In these cases all such samples $w$ must be removed from $\cM^{\rm min}$ (respectively, $\cM^{\max}$). We call this procedure pruning of the sets.

All that is left is to explain the procedure of generation of new samples. First of all, let $\A = \{z | \cM^{\rm min} \preceq z \preceq \cM^{\rm max} \}$ with the Lebesgue measure $|\A|$. The measure of the whole space, which is the defined above hypercube $\B$, is its volume $|\B|$. The value $|\A|$ represents the current error of the algorithm. In order to sample from the actual level set $y = g(z)$ the measure $|\A|$ has to be equal to zero. 

We consider two ways of generating samples: random and greedy. In ``random'' generation we randomly generate points in the set $\A$. Let $z = \begin{pmatrix} z_1 & \dots & z_n \end{pmatrix}$ be a sample, which we need to generate. First, we generate the scalars $z_1$, $\dots$, $z_{n-1}$ on the edge the hypercube $\B$ with $z_n =b^0_n$. We do it by using a probability distribution $\delta_2$ with the support on the whole edge. Then we compute the limits for the generation of $z_n$ such that $\cM^{\rm min}\preceq_\cK z \preceq_\cK \cM^{\rm max}$. Next, we generate $z_n$ using the same distribution $\delta_2$, while adjusting its support. We generate $N_{\rm rn}$ samples in this manner. 

The ``greedy'' generation is meant to exploit more information about $|\A|$ and decrease the value $|\A|$, which represents the error of the algorithm. For every sample $z^i$ in $\cM^{\rm min} \bigcup \cM^{\rm max}$ we find points $w^j$ from $\cM^{\rm min} \bigcup \cM^{\rm max}$ such that $w^j \succeq z^i$ and such that there is no other $\xi$ satisfying $w^j\succeq \xi \succeq z^i$. Then we form hypercubes $\D_{j i}$ with vertices $w^j$ and $z^i$ and choose $\D_{j i}$ with the largest volume. We denote these hypercubes $\D_i$. Having done it for every sample $z^i$ in $\cM^{\rm min}\bigcup \cM^{\rm max}$ we compute $N_{\rm gr}$ hypercubes $\D_{i}$ with largest volumes. After that we generate $N_{\rm gr}$ samples in each of the hypercubes $\D_{i}$ using a probability distribution $\delta_1$. If no information about the system is known then we choose $\delta_1$, $\delta_2$ as uniform distributions. Otherwise, we can choose a Beta distribution so that the median approximates the most probable location of the curve. 

The procedure is summarized in Algorithm~\ref{alg:s1}. Note that our sample generation guarantees that the measure $|\A|$ is nonincreasing with generation of new samples, while random sampling guarantees that the value $|\A|$ eventually converges to zero with the rate of convergence depending on specific functions $g$. However, in our examples we observed exponential empirical convergence, which can be explained by the fact that our algorithm can be seen as a bisection procedure using random sampling.

\begin{algorithm}[t]
\caption{Computation of a level set of incomparable points of a function $\alpha = g(x)$}
\label{alg:s1}
\begin{algorithmic}[1]
\State {\bf Inputs:} Oracle $\cO$, the number of greedy samples $N_{\rm gr}$, the number of random samples $N_{\rm rn}$, the total number of samples $N_{\rm tot}$, the boundary points $b^0$, $b^1$
\State {\bf Outputs:} The sets of points $\cM^{\rm min}$, $\cM^{\rm max}$.
     \State Set $N = N_{\rm gr}+N_{\rm rn}$
     \State Adjust the boundary points $b^0$ and $b^1$
     \State Set $\cM^{\rm max}=\{c^i\}$, $\cM^{\rm min} = \{d^i\}$
     \For{$i= 1,\dots, [N_{\rm tot}/N]$}
	 \State Compute the hypercubes $\D_{i}$, where $i = 1, \dots, N_{\rm gr}$ 
	 \State Generate $N_{\rm gr}$ samples $z^j$ in the hypercubes $\D_{i}$ using a probability distribution $\delta_1$.
	 \State Generate $N_{\rm rn}$ samples $z^j$ in the admissible set $\A$ using a probability distribution $\delta_2$. 
     \For{$j = 1,\dots, N$}
			\State If $\cO(z^j) = 0$ add the sample $z^j$ to the set $\cM^{\rm min}$, otherwise add $z^j$ to the set $\cM^{\rm max}$.
     \EndFor
     \State Prune the sets $\cM^{\rm min}$, $\cM^{\rm max}$ for comparable samples $z^j$ and update the admissible set $\A$.
     \EndFor
\end{algorithmic}
\end{algorithm}

\subsubsection{Oracles for Computation of Isostables and Basins of Attraction}
 Without loss of generality we assume that the stable equilibrium lies at the origin. Since the sets $\partial \cB_\alpha(0)$ and $\partial \cB(0)$ are level sets of some increasing function $y = g(x)$ according to Proposition~\ref{prop:level-set-proper}, we can use Algorithm~\ref{alg:s1} to compute them. In order to do so, we can use the following oracle:
\begin{gather}
\cO(x) = \begin{cases}
0  & \textrm{ if }s_1(x) < \alpha \textrm{ and } \|\phi(T,x) - x^\ast\| < \varepsilon, \\
1  & \textrm{ otherwise},
\end{cases} 
\end{gather}
where the value of the eigenfunction $s_1(x)$ is computed using Laplace averages~\eqref{laplace-average}, and $T$ is a large enough time. We need the second condition $\|\phi(T,x) - x^\ast\| < \varepsilon$ to make sure that the points also lie in $\cB(0)$. Even though it is theoretically unlikely to have a point $x$ with a finite (nonzero) $s_1(x)$ not lying in $\cB(0)$, such situations can occur numerically. If we need to compute $\cB(0)$, then we can drop the first condition $s_1(x) < \alpha$, since in this case $\alpha = \infty$. 

We can also apply this algorithm to the computation of the so-called switching separatrix defined in the context of pulse-controlled bistable systems~\cite{sootla2015pulsesaut}, since it is also the level set of an increasing function.

\section{Numerical Examples \label{s:examples}}
\subsection{ A Two-State LacI-TetR Switch} 
\begin{figure}[t]\centering
  \includegraphics[width = 0.85\columnwidth]{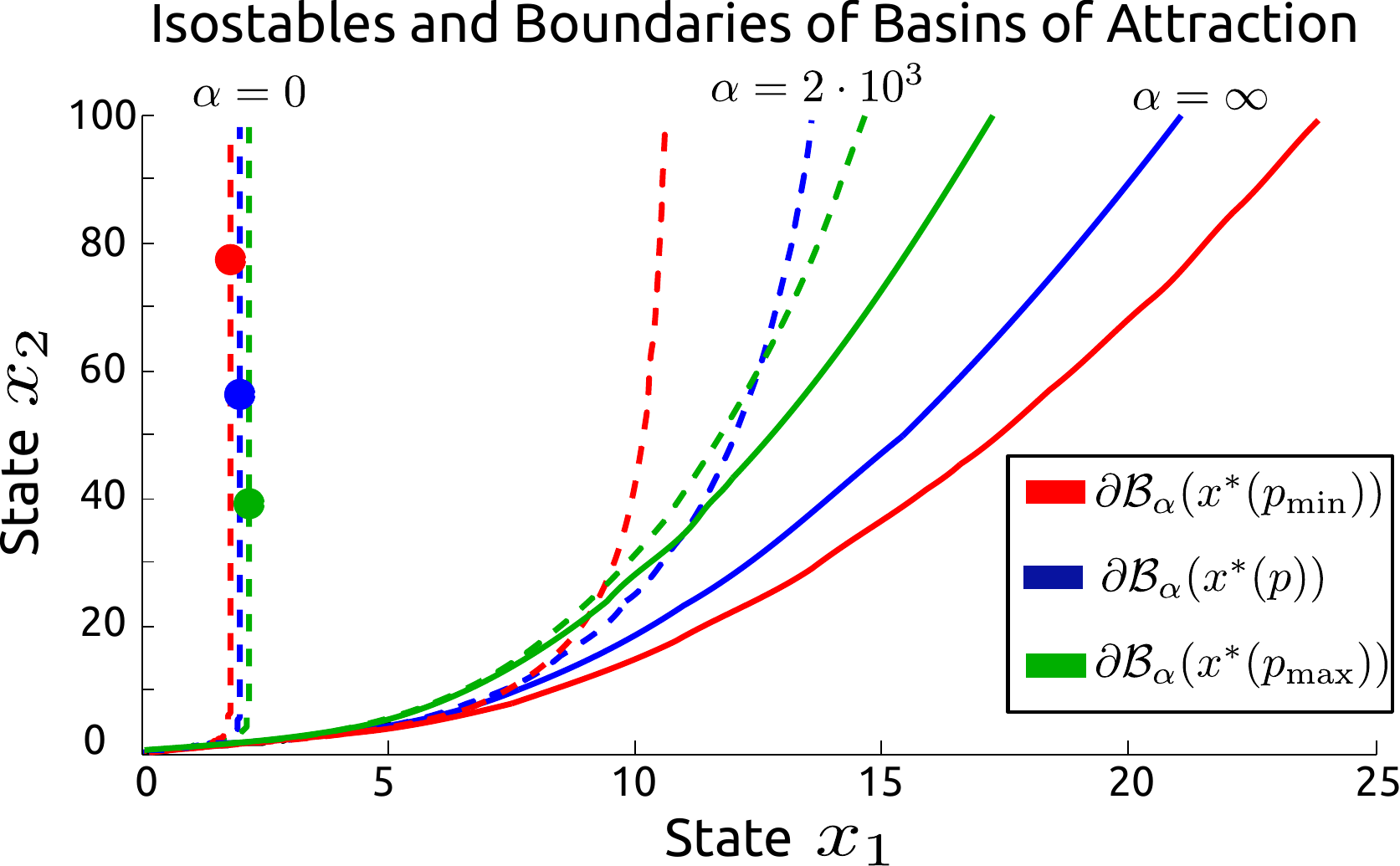}
\caption{Isostables $\partial \cB_{0}(x^\ast(q))$, $\partial \cB_{2\cdot 10^3}(x^\ast(q))$, and boundaries of basins of attraction $\partial \cB(x^\ast(q))$ for $q$ equal to $p$, $p_{\rm min}$, or $p_{\rm max}$. The dots represent the equilibria $x^\ast(p)$ (blue), $x^\ast(p_{\rm min})$ (red), $x^\ast(p_{\rm max})$ (green) for different parameter values} \label{fig:toggle-2d-iso}
\end{figure}
The genetic toggle switch is composed of two mutually repressive genes \emph{LacI} and \emph{TetR} and was a pioneering genetic system for synthetic biology~\cite{Gardner00}. We consider the following model of the toggle switch
\begin{gather}\label{sys:toggle-2d}
\begin{aligned}
\dot x_1 &= p_{1 1} + \frac{p_{1 2}}{ 1 + x_2^{p_{1 3}}} - p_{1 4} x_1, \\
\dot x_2 &= p_{2 1} + \frac{p_{2 2}}{ 1 + x_1^{p_{2 3}}} - p_{2 4} x_2,
\end{aligned}
\end{gather}
where all $p_{i j} \ge 0$. The states $x_{i}$ represent the concentration of proteins LacI and TetR, whose mutual repression is modeled via a rational function. The parameters $p_{1 1}$ and $p_{2 1}$ model the basal synthesis rate of each protein. The parameters $p_{1 4}$ and $p_{2 4}$ are degradation rate constants, and $p_{1 2}$, $p_{2 2}$ describe the strength of mutual repression. The parameters $p_{1 3}$, $p_{2 3}$ are called Hill coefficients. By means of Proposition~\ref{prop:kamke} we can readily check that the model is monotone on $\Rnn^2$ for all nonnegative parameter values with respect to the orthant $\diag{\begin{pmatrix} 1 & -1 \end{pmatrix}} \Rnn^2$. Moreover, the model is monotone with respect to all parameters but $p_{1 3}$, $p_{2 3}$. The stable equilibrium $x^\ast$ has the state $x_2$ ``switched on'' (i.e., $x_2$ is much larger than $x_1$), while $x^\bullet$ has the state $x_1$ ``switched on'' (i.e., $x_1$ is much larger than $x_2$).

In order to evaluate Theorem~\ref{thm:ba} we consider the set of admissible parameters $\cP = \{q | p_{\rm max} \succeq_p q \succeq_p p_{\rm min}\}$, where
\begin{gather*}
 p_{\rm min} = \begin{pmatrix}
 1.8 & 950 & 4 & 1 \\ 1.2 & 1050 & 3 & 2
 \end{pmatrix} \,\,
 p_{\rm max} = \begin{pmatrix}
 2.2 & 1100 & 4 & 1 \\ 0.7 & 900 & 3 & 2 
 \end{pmatrix}.  
\end{gather*}
We compute the isostables $\partial \cB_\alpha(x^\ast(\cdot))$ with $\alpha = 0$, $2\cdot 10^3$, $\infty$ for systems with parameters $p$, $p_{\rm min}$, $p_{\rm max}$, where $\partial \cB_\infty(x^\ast(\cdot))$ is the boundary of the basin of attraction $\cB(x^\ast(\cdot))$, and the matrix of parameters $p$ is equal to:
\begin{gather*}
 p = \begin{pmatrix}
 2 & 1000 & 4 & 1 \\ 1 & 1000 & 3 & 2 
 \end{pmatrix}.
\end{gather*}
 The computational results are depicted in Figure~\ref{fig:toggle-2d-iso}. These results suggest that for all parameter values $p\in\cP = \{q | p_{\rm max} \succeq_{\cK_p} q \succeq_{\cK_p} p_{\rm min} \}$ the manifold $\partial \cB(x^\ast(p))$ will lie between the manifolds $\partial \cB(x^\ast(p_{\rm min}))$ and $\partial \cB(x^\ast(p_{\rm max}))$. It appears that $\partial \cB_0(x^\ast(p))$ also lies between the manifolds $\partial \cB_0(x^\ast(p_{\rm min}))$ and $\partial \cB_0(x^\ast(p_{\rm max}))$, however, in a different order. This change of order and continuity of $s_1$ implies that there exists an $\alpha$ such that at least two manifolds $\partial \cB_\alpha(x^\ast(p))$, $\partial \cB_\alpha(x^\ast(p_{\rm max}))$, $\partial \cB_\alpha(x^\ast(p_{\rm min}))$ intersect. This case is also depicted with $\alpha = 2\cdot 10^3$. This observation implies that $s_1(x,p)$ is not an increasing function in $p$. This is consistent with the linear case, where changes in the drift matrix $A$, will simply rotate the hyperplane $\left\{ x \bigl| w_1^T x = 0 \right\}$, where $w_1$ is the left dominant eigenvector of $A$.
\begin{figure}[t]\centering
  \includegraphics[width = 0.85\columnwidth]{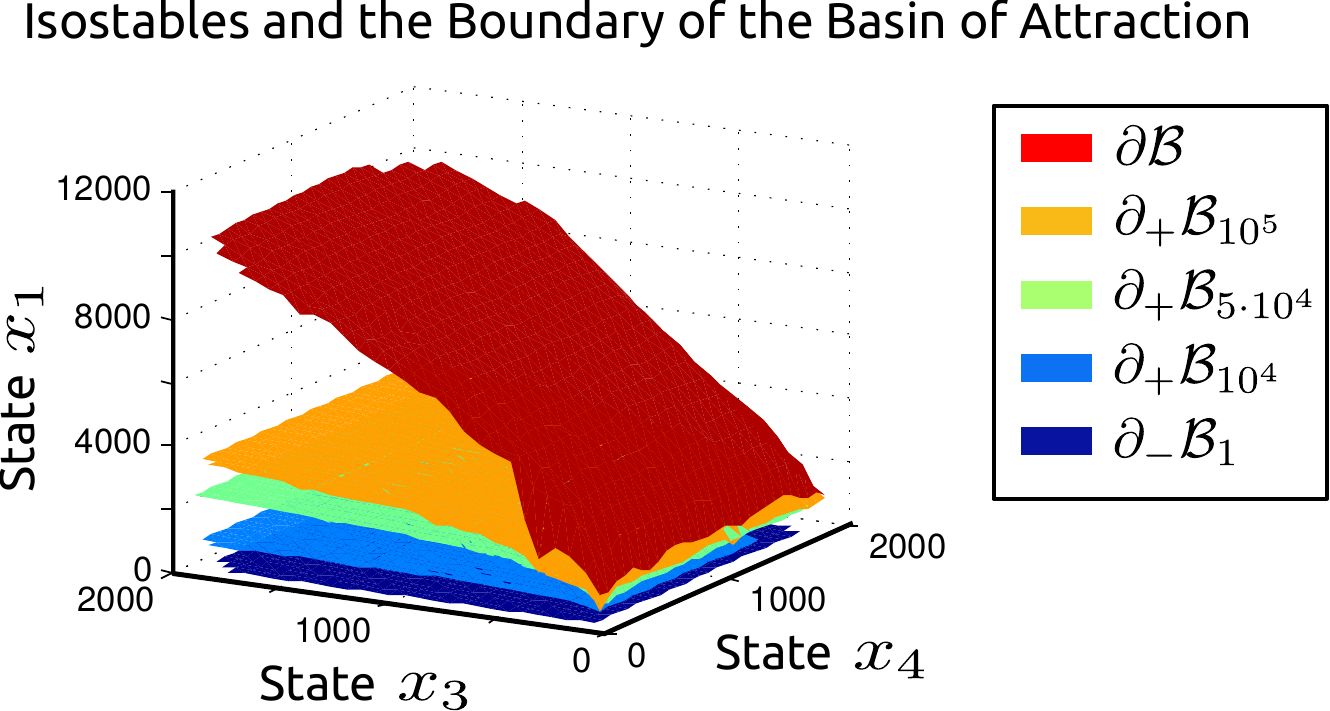}
\caption{Isostables and the basin of attraction of a four state toggle switch. We depict cross section of isostables $\partial \cB_{10^4}$, $\partial \cB_{2 \cdot 10^4}$ and $\partial \cB_{10^5}$, and the boundary of the basin of attraction $\partial \cB$ for a four state system with $x_2 = 10^{-7}$.} \label{fig:toggle-3d-iso}
\end{figure}
 
\subsection{A Four State Toggle Switch.} 
We can compute the basins of attraction and isostables for $n$ dimensional models, as well. Naturally, we can visualize only $3$-D cross-sections of basins of attraction of $n$ dimensional systems. We consider again a toggle switch, while modeling also mRNA concentrations activating their corresponding proteins. The resulting model is as follows
\begin{gather}\label{sys:toggle-3d}
\begin{aligned}
\dot x_1 &=  p_{1 1} x_2^{p_{1 2}} - p_{1 3} x_1, \quad \dot x_2 &= \frac{p_{2 1}}{ 1 + x_3^{p_{2 2}}} - p_{2 3} x_2, \\  
\dot x_3 &= p_{3 1} x_4^{p_{3 2}} - p_{3 3} x_3,  \quad \dot x_4 &=  \frac{p_{4 1}}{ 1 + x_1^{p_{4 2}}} - p_{4 3} x_4, 
\end{aligned}
\end{gather}
 where the parameters have the following values 
\begin{gather*}
p_{1 1} = 1000,~ p_{1 2}= 1,~p_{1 3} = 1,~p_{2 1} = 2,~ p_{2 2}= 3,~p_{2 3} = 1\\
p_{3 1} = 1000,~ p_{3 2}= 1,~p_{3 3} = 2,~p_{4 1} = 1,~ p_{4 2}= 3,~p_{4 3} = 2,
 \end{gather*}

Again, by means of Proposition~\ref{prop:kamke}, we can check that the model is monotone on $\Rnn^4$ for all nonnegative parameter values with respect to the orthant $\diag{\begin{pmatrix} 1 & 1 & -1 & -1\end{pmatrix}} \Rnn^4$. We will plot a cross-section with $x_2 = 10^{-7}$. In Figure~\ref{fig:toggle-3d-iso}, we plot isostables $\partial_- \cB_{1}$, $\partial_+ \cB_{10^4}$, $\partial_+ \cB_{5 \cdot 10^4}$ and $\partial_+ \cB_{10^5}$. We also plot the boundary of the domain of attraction. 

It is noticeable that the sets $\partial\cB_\alpha$ do not depend on the state $x_4$ in a significant manner. Hence the state $x_4$ can be reduced using time-scale separation methods. Similarly, we can reduce $x_2$. This is in an agreement with ``a rule of thumb'' in protein-mRNA interactions, which states that mRNA dynamics are much faster than protein dynamics and hence can be reduced. Moreover, for every fixed $x_2$, $x_4$ the shapes of isostables are similar to the shape of the isostables of a two state toggle switch. 
 
\section{Conclusion}
In this paper, we study geometric properties of monotone systems such as properties of isostables and basins of attraction. Isostables contain the initial states of trajectories that converge synchronously toward the equilibrium and can be computed using the Koopman operator framework. Isostables are boundaries of forward-invariant sets of dynamical systems, and hence serve as convenient refinement of the basin of attraction notion. We show that the isostables and the boundary of basin of attraction of monotone systems share the same geometric properties. For example, these manifolds do not contain strictly comparable points in the order. Our derivations are based on a positivity result for the Koopman semigroup, which is related to the statement of the Perron-Frobenius theorem. We then focus on properties of basins of attraction of bistable systems. We show how to estimate basins of attraction of monotone parameter-dependent systems subject to parametric uncertainty, and illustrate our results on numerical examples. Future work includes more practical implications of our results in synthetic biology, and a further study of the properties of our algorithm such as rate of convergence, efficiency and scalability.

\bibliography{ba_acc_arxiv.bbl}
\end{document}